\documentclass[a4paper]{article}

\usepackage{amsmath, amssymb, amsfonts}
\usepackage{theorem}

\newtheorem{Proposition}{Proposition}[section]
\newtheorem{Lemma}[Proposition]{Lemma}

\newtheorem{Theorem}[Proposition]{Theorem}

\newtheorem{Observation}[Proposition]{Observation}

\renewcommand{\deg}{d}

\newcommand{\emtext}[1]{\text{\em #1}}
\newcommand{\COMMENT}[1]{}
\def\minor{\preccurlyeq}
\def\Minor{\succcurlyeq}

\def\P{\mathcal{P}}
\def\Q{\mathcal{Q}}
\def\RR{\mathcal{R}}
\def\T{\mathcal{T}}
\def\interior{\mathaccent"7017\relax}

{\theorembodyfont{\rmfamily}

}
\newcommand{\qed}{\hbox{\rule{6pt}{6pt}}}
\newcommand{\D}{\ensuremath{\mathcal{D}}}
\newcommand{\V}{\ensuremath{\mathcal{V}}}
\newcommand{\W}{\ensuremath{\mathcal{W}}}
\newcommand{\boundary}{\partial}

\newenvironment{proof}%
{\noindent{\bf Proof.}\ }%
{\unskip\nobreak\hfill\penalty50\hskip2em\hbox{}\nobreak\hfill%
       $\square$\parfillskip=0pt\finalhyphendemerits=0\par\bigskip}%

\newcommand\N{\mathbb N}
\newcommand{\R}{\mathbb R}
\newcommand{\td}{tree-decom\-po\-si\-tion}

\advance\textheight by -2cm% to make page fit on my monitor at 2000

\title{{\bf The Erd\H{o}s-P\'osa property for clique minors\\ in highly connected graphs}}

\author{
{\bf Reinhard Diestel}
\\
\and
{\bf Ken-ichi Kawarabayashi}\thanks{National Institute of Informatics, 2-1-2, Hitotsubashi, Chiyoda-ku, Tokyo, Japan}\thanks{Research is partly
supported by Japan Society for the
Promotion of Science, Grant-in-Aid for Scientific Research,
by C \& C Foundation, by
Kayamori Foundation and by Inoue Research Award
for Young Scientists.}\thanks{Email address: {\tt k\_keniti@nii.ac.jp}}
\\
\and
{\bf Paul Wollan}
\thanks{
University of Rome, ``La Sapienza"
Department of Computer Science,
Via Salaria 113,
Rome 00198 Italy
}
\thanks{Email address: {\tt paul.wollan@gmail.com}}
}

\begin{document}

\date{}

\maketitle

\begin{abstract}

We prove the existence of a function $f\colon \N^2\to\N$ such that, for all $p,k\in\N$, every $(k(p-3) + 14p+14)$-connected graph either has $k$ disjoint $K_p$ minors or contains a set of at most $f(p,k)$ vertices whose deletion kills all its $K_p$ minors. For fixed~$p \ge 5$, the connectivity bound of about $k(p-3)$%
   \COMMENT{}
   is smallest possible, up to an additive constant: if we assume less connectivity in terms of~$k$, there will be no such function~$f$.
\end{abstract}


\section{Introduction}

A set of graphs $\mathcal{C}$ has the \emph{Erd\H{o}s-P\'osa property} if there exists a function $f = f(k)$
such that for all $k \ge 1$, any graph $G$ either contains $k$ vertex disjoint subgraphs in $\mathcal{C}$,
or there exists a subset of vertices $X \subseteq V(G)$ with $|X| \le f(k)$
such that every subgraph of $G$ in $\mathcal{C}$
intersects a vertex of $X$.  The name derives from an article of Erd\H{o}s and P\'osa \cite{erdosposa}
where they show that the set $\mathcal{C}$ of cycles has this property.

Let $G$ and $X$ be graphs.  An \emph{extension of $X$} is a graph that can be contracted
to $X$.  An \emph{instance of an $X$-minor} in~$G$ is a subgraph $H$ of $G$ isomorphic
to an extension of $X$.  The set $\mathcal{C}$ of cycles can be thought of as the set
of extensions of $K_3$, the complete graph of three vertices.  Thus
the result of Erd\H{o}s and P\'osa can be reformulated as follows:  there exists a function $f(k)$ such that
any graph $G$ either contains $k$ disjoint instances of $K_3$ as a minor, or there exists a subset of vertices
$X\subseteq V(G)$ with $|X| \le f(k)$ such that $G-X$ does not contain $K_3$ as a minor.
For any graph $H$, let  $\mathcal{C}_H$ be the set of extensions of $H$.
Robertson and Seymour \cite{RS5} have exactly characterized which graphs $H$
have the property that the set $\mathcal{C}_H$ has the Erd\H{o}s-P\'osa property:  the set
$\mathcal{C}_H$ has the Erd\H{o}s-P\'osa property if and only if
$H$ is planar.

The purpose of this article is to prove the following theorem.
\begin{Theorem}\label{thm:main}
There exists an $\N^2\to\N$ function $f$ such that, for all $p,k\in\N$, every\penalty-10000\
${(k(p-3) + 14p+14)}$-connected graph $G$ either contains $k$ disjoint instances of a $K_p$-minor or has a set $X$ of at most $f(p,k)$ vertices such that $G-X$ has no $K_p$-minor.
\end{Theorem}

We also show that the connectivity bound in Theorem~\ref{thm:main} is best possible, up to an additive constant, for fixed $p\ge 5$. Indeed for each~$p$ we shall find a constant $c_p$ such that for all $k,n\in\N$ there are $(k(p-3)-c_p)$-connected graphs that do not contain $k$ disjoint instances of $K_p$ as a minor but in which no set of at most $n$ vertices kills all their $K_p$ minors. Hence it is not possible to define a function $f(p,k)$ as in Theorem~\ref{thm:main} that makes the theorem true for all $(k(p-3)-c_p)$-connected graphs.

We will need the following definitions.  We write $X\minor G$ to express that $X$ is a minor of~$G$.
Given an extension $H$ of an $X$ minor in $G$, a \emph{branch set} the $X$ minor is a maximal subset of
vertices of $H$ which is contracted to a single vertex when contracting $H$ to $X$.  By $kX$ we denote the disjoint union of $k$ copies of a graph~$X$. A path starting in $x\in X$ and ending in $y\in Y$ is an \emph{$X$--$\,Y$ path} if $x$ is its only vertex in $X$ and $y$ is its only vertex in~$Y$.%
   \COMMENT{}
   A~set $\P$ of disjoint paths is a \emph{linkage}. If it consists of $X$--$\,Y$ paths and these meet all of~$X\cup Y$, it is an \emph{$X$--$Y$ linkage}. (Then $|X|=|\P|=|Y|$.) Two linkages $\mathcal{P}$ and $\mathcal{Q}$ of the same order
are \emph{equivalent} if for every $P \in \mathcal{P}$ there exists a (\emph{corresponding}) path $Q \in \mathcal{Q}$
such that $P$ and $Q$ have the same endpoints.

We recall that a \emph{tree decomposition} of a graph $G$ is a pair $(T, \mathcal{W})$
where $T$ is a tree and $\mathcal{W} = \{W_t \subseteq V(G) : t \in V(T)\}$ is a collection
of subsets of the vertices of $G$ indexed by the vertices of $T$.  Moreover,
the collection of subsets $\mathcal{W}$ satisfies the following:
\begin{itemize}
\item $\bigcup_{t \in V(T)} W_t = V(G)$,
\item for every edge $e = uv$ in $E(G)$, there exists $t \in V(T)$ such that $v, u \in W_t$, and
\item for all $v \in V(G)$, the vertices $\{t \in V(T) : v \in W_t\}$ induce a connected
subtree of $T$.
\end{itemize}
The \emph{width} of the decomposition $(T, \mathcal{W})$ is $\max_{t \in V(T)} |W_t| -1$, and
the \emph{tree-width} of a graph $G$ is the minimum width of a tree decomposition of $G$.  A \emph{path
decomposition} is simply a tree decomposition where the graph $T$ is a path.  Given
a path decomposition $(P, \mathcal{W})$ where the vertices of $P$ are $v_1, v_2, \dots, v_k$ and
occur in that order on the path, we will often simplify the notation and refer to the path decomposition as $(W_1, W_2, \dots, W_k)$ where $W_{v_i} =: W_i$ for $1 \le i \le k$.

For any further notions not covered here we refer to~\cite{DiestelBook05}.

The paper is structured as follows. We begin in Section~\ref{sec:smalltw} by proving our theorem for graphs of small tree-width. For graphs of large tree-width we shall use a structure theorem or Robertson
and Seymour, although we will follow the notation and statement of~\cite{structure_thm}; this is explained in Section~\ref{sec:structure}. At the end of Section~\ref{sec:structure} we give a more detailed overview of how the proof then proceeds until the end of Section~\ref{lastsection}. In Section~\ref{sec:tight} we give our construction showing that the connectivity bound in Theorem~\ref{thm:main} is tight.

\section{Proof of Theorem \ref{thm:main}}\label{sec:smalltw}

The proof of Theorem \ref{thm:main} proceeds by considering separately the cases of when
the tree-width of the graph is large or small. In this, we follow much of the recent work analyzing the existence of clique minors in large graphs.  See \cite{BKMM, JORG1, JORG2, ft}.  The bounded tree-width case is easy:

\begin{Theorem}\label{thm:rsep}
For every $w\in\N$ there is a function $f_w\colon \N^2\to\N$ such that, for all $p,k\in\N$, every graph $G$ of tree-width~$< w$ either contains $k$ disjoint instances of a $K_p$ minor or has a set $X$ of at most $f_w(p,k)$ vertices such that $G-X$ has no $K_p$-minor.
\end{Theorem}

\begin{proof}
For fixed $w$ and~$p$ we define $f_w (p,k)$ recursively for $k=1,2,\dots$. Clearly, $f_w (p,1) := 0$ satisfies the theorem for $k=1$. Given $k\ge 2$, let
 $$f_w(p,k) := 2 f_w(p,k-1) + w\,.$$
 To see that this satisfies the theorem, let $G$ be given, with a \td\ $(T,(V_t)_{t\in T})$ of width $<w$. Direct the edges $t_1 t_2$ of $T$ as follows.
Let $T_1,T_2$ be the components of $T-t_1 t_2$ containing $t_1$ and $t_2$,
respectively, and put
 $$G_1 := G[\bigcup_{t\in T_1} (V_t\setminus V_{t_2})]\quad {\rm and}\quad
   G_2 := G[\bigcup_{t\in T_2} (V_t\setminus V_{t_1})]\,.$$
 Direct the edge $t_1 t_2$ towards $G_i$ if $G_i$ has a $K_p$-minor,
thereby giving $t_1 t_2$ either one or both or neither direction.

If every edge of $T$ receives at most one direction, we follow
these to a node $t\in T$ such that no edge at $t$ in $T$ is directed away
from~$t$. As $K_p$ is connected, this implies that $V_t$
meets every instance of a $K_p$ minor in~$G$ \cite[Lemma 12.3.1.]{DiestelBook05}. This completes the proof with $X = V_t$, since
$|V_t| \le w\le f_w(p,k)$ by the choice of our \td.

Suppose now that $T$ has an edge $t_1 t_2$ that received both directions.
For each $i=1,2$ let us ask if $G_i$ has a set $X_i$ of at most $f_w(p,k-1)$ vertices such that $G_i - X_i$ has no $K_p$-minor. If this is the case for both~$i$, then as earlier there is no $K_p$-minor in $G-X$ for $X:= X_1\cup X_2\cup (V_{t_1}\cap V_{t_2})$.

Suppose then that~$G_1$, say, has no such set $X_1$ of vertices. By the induction
hypothesis, $G_1$~contains $(k-1)$ disjoint instances of a~$K_p$-minor. Since $t_1 t_2$ was
also directed towards~$t_2$, there is another such instance in~$G_2$. This
gives the desired total of $k$ disjoint instances of a $K_p$-minor in~$G$.
\end{proof}

The bulk of the work in proving Theorem \ref{thm:main} will be the case of large tree-width:

\begin{Theorem}\label{thm:bigtw}
For all $p,k\in\N$ there exists $w = w(p,k)\in\N$ such that every $(k(p-3) + 14p+14)$-connected graph of tree-width at least $w$ contains $k$ disjoint instances of a~$K_p$ minor.
\end{Theorem}

\noindent{\bf Proof of Theorem \ref{thm:main}, assuming Theorems \ref{thm:rsep} and~\ref{thm:bigtw}.}
   Given $p,k\in\N$ define $f(p,k) := f_w (p,k)$, where $w = w(p,k)$ is provided by Theorem \ref{thm:bigtw} and $f_w$ by Theorem~\ref{thm:rsep}. Let $G$ be a $(k(p-3) +14p+14)$-connected graph. If $G$ has tree-width~$<w$, the assertion which Theorem~\ref{thm:main} makes about $G$ is tantamount to that of Theorem~\ref{thm:rsep}. If $G$ has tree-width at least~$w$, the assertion follows from Theorem \ref{thm:bigtw}.
   \hfill\qed

\bigbreak

Given Theorem \ref{thm:bigtw}, one might ask if a stronger statement might be true:  whether there exists a constant $c$ such that every sufficiently large $(k(p-3) + cp)$-connected graph contains $k$ disjoint instances of a $K_p$ minor.  However, the bound on the connectivity is not sufficient for such a strengthening.  Consider the complete bipartite graph $K_{k(p-2) + cp, T}$ for large values of $T$ and some fixed constant $c$.  Such a graph cannot contain $k$ disjoint instances of $K_p$ as a minor (assuming $k$ is chosen to be at least $cp+1$).  However, the graph has tree-width $k(p-2)+cp$, i.e. the tree-width is bounded with respect to $k$ and $p$.
\COMMENT{}

The proof of Theorem~\ref{thm:bigtw} will occupy us until the end of Section~\ref{lastsection}. Let $p,k\in\N$ be given, {\em and fixed until the end of the proof of Theorem~\ref{thm:bigtw}.} Several parameters defined in the course of the proof will depend implicitly  on this choice of $p$ and~$k$.

Given positive integers $\ell$ and~$n$, let us define the \emph{$\ell$-ladder} $L(\ell)$ and the \emph{fan} $F(\ell,n)$ as follows. Let $P = u_1\dots u_\ell$ and $Q = v_1\dots v_\ell$ be disjoint paths, and let $L(\ell)$ be obtained from their union by adding all the edges~$u_i v_i$. To obtain $F(\ell,n)$ from~$L(\ell)$, add $n$ independent vertices $w_1,\dots,w_n$, and join each of these to all the vertices of~$Q$.

It is easy to see that $F(p, p-3)$ has a $K_p$ minor: with $p-3$ two-vertex branch sets of the form $\{v_i,w_i\}$, and three further branch sets $\{v_{p-2}\}$, $\{v_{p-1}\}$, and $\{v_p, u_p, u_{p-1}, u_{p-2}\}$. Consequently, $F(kp, k(p-3))$ contains $k$ disjoint instances of a $K_p$ minor. It will thus suffice for our proof of Theorem~\ref{thm:bigtw} to find a $F(kp, k(p-3))$-minor in the graph under consideration.

%%%%%%%%%%%%%%%%%%%%%%%%%%%%%%%%%%%%%%%%%%%%%%

\section{The excluded minor theorem}\label{sec:structure}

In this section, we present a structure theorem for graphs with no large clique minor of
Robertson and Seymour \cite{RS17}.  We follow the notation and exact statement in \cite{structure_thm}.

A \textit{vortex} is a pair $V = (G,\Omega)$, where $G$ is a graph and $\Omega =: \Omega(V)$ is a linearly ordered set $(w_1, \ldots, w_n)$ of vertices in~$G$. These vertices are the \textit{society vertices} of the vortex; the number $n$ is its \textit{length}. We do not always distinguish notationally between a vortex and its underlying graph; for example, a \textit{subgraph of $V$} is just a subgraph of~$G$.  Also, we will often use $\Omega$ to refer both to the linear order of the vertices
$w_1, \dots, w_n$ as well as the set of vertices $\{w_1, \dots, w_n\}$.

A path--decomposition $\D = (X_1,\ldots,X_m)$ of $G$ is a \textit{decomposition of~$V$} if $m=n$ and $w_i\in X_i$ for all~$i$. The \textit{depth} of the vortex~$V$ is the minimum width of a path--decomposition of $G$ that is a decomposition of~$V$.

The \emph{adhesion} of our decomposition $\D$ is the maximum value of ${|X_{i-1}\cap X_{i}|}$, taken over all $1<i\leq n$. Write $Z_i:=(X_{i-1}\cap X_{i})\backslash \Omega$, for all $1 < i\leq n$.  Then, $\D$ is \textit{linked} if
\begin{itemize}
\item[$i.$] all these $Z_i$ have the same size;
\item[$ii.$] there are $|Z_i|$ disjoint $Z_i$--$Z_{i+1}$ paths in~$G[X_i] - \Omega$, for all $1<i<n$;
\item[$iii.$] $X_i \cap \Omega =\{w_i,w_{i+1}\}$ for all $i=1,\dots,n$, where $w_{n+1}:=w_n$.
\end{itemize}

Note that the union of those $Z_i$--$Z_{i+1}$ paths is a disjoint union of $X_1$--$X_n$ paths in~$G$; we call the set of these paths a \emph{linkage} of~$V$ with respect to $(X_1,\ldots,X_m)$.
We define the \textit{(linked) adhesion} of a vortex to be the minimum adhesion of a (linked) decomposition of that vortex; if it has no linked decomposition, its \textit{linked adhesion} is infinite.

For a positive integer $\alpha$, a graph $G$ is \textit{$\alpha$--nearly embeddable} in a surface $\Sigma$ if there is a subset $A\subseteq V(G)$ with $|A|\leq \alpha$ such that there are natural numbers  $\alpha'\leq \alpha$ and $n\ge \alpha'$ for which $G-A$ can be written as the union of $n+1$ graphs $G_0,\ldots,G_n$ such that the following holds:
\begin{enumerate}%[(i)]
\item[$i.$]  For all $1 \le i \le n$ and $\Omega_i:= V(G_i\cap G_0)$, the pair $(G_i,\Omega_i)=:V_i$ is a vortex, and for $1 \le i < j \le n$, $G_i\cap G_j \subseteq G_0$.
\item[$ii.$] The vortices $V_1,\ldots,V_{\alpha'}$ are disjoint and have adhesion at most $\alpha$; we denote this set of vortices by $\V$.
\item[$iii.$] The vortices $V_{\alpha'+1},\ldots,V_n$ have length at most 3; we denote this set of vortices by $\W$.
\item[$iv.$] There are closed discs in~$\Sigma$ with disjoint interiors $D_1,\ldots, D_n$ and an embedding ${\sigma: G_0 \hookrightarrow \Sigma-\bigcup_{i=1}^n D_i}$ such that $\sigma(G_0)\cap\boundary D_i = \sigma(\Omega_i)$ for all~$i$ and the generic linear ordering of $\Omega_i$ is compatible with the natural cyclic ordering of its image (i.e., coincides with the linear ordering of $\sigma(\Omega_i)$ induced by $[0,1)$ when $\boundary D_i$ is viewed as a suitable homeomorphic copy of $[0,1]/\{0,1\}$). For $i=1,\dots,n$ we think of the disc $D_i$ as \emph{accommodating} the (unembedded) vortex~$V_i$, and denote $D_i$ as~$D(V_i)$.\sloppy
\end{enumerate}
We call $(\sigma,G_0,A,\V,\W)$ an \emph{$\alpha$--near embedding} of $G$ in $\Sigma$.

Let $G'_0$ be the graph resulting from $G_0$ by joining any two unadjacent vertices $u,v\in G_0$ that lie in a common vortex $V\in\W$; the new edge $uv$ of $G'_0$ will be called a \emph{virtual edge}. By embedding these virtual edges disjointly in the discs $D(V)$ accommodating their vortex~$V$, we extend our embedding $\sigma\colon G_0\hookrightarrow\Sigma$ to an embedding $\sigma'\colon G'_0\hookrightarrow\Sigma$. We shall not normally distinguish $G'_0$ from its image in $\Sigma$ under~$\sigma'$.

The more widely known version of the excluded minor theorem of Robertson and Seymour (\cite{RS16}, see also \cite{DiestelBook05}) decomposes a graph not containing a fixed $H$ as a minor into a tree-like
structure of $\alpha$-nearly embeddable graphs, where the value of $\alpha$ depends solely on the
graph $H$.  We will need a variation of the structure theorem which ensures both that the
vortices are linked and that there is a large grid-like graph embedded in the surface when the graph is
assumed to have large tree width.

A vortex $(G_i,\Omega_i)$ is \emph{properly attached} to $G_0$ if, for every pair of distinct vertices $x,y\in \Omega_i$, there is a path $P_{xy}$ in $G_i$ with endvertices $x$ and $y$ and all inner vertices in $G_i-\Omega_i$ and further, for every choice of three distinct vertices $x,y,z\in\Omega_i$, the paths $P_{xy}$ and $P_{yz}$ can be chosen internally disjoint.

The \textit{distance} of two points $x,y\in \Sigma$ is the minimal value of $|G\cap C|$ taken over all curves $C$ in the surface that link $x$ and $y$ and hit the graph in vertices only.
The \textit{distance} of two vortices $V$ and $W$  is the minimal distance of a point $v\in D(V)$ and a point $v'\in D(W)$.

When a graph is embedded in a surface, a topological component of the surface minus the
graph that is homeomorphic to a disc is a \emph{face}.  The \emph{outer cycle} of a 2-connected plane graph is the cycle bounding its infinite face.  A cycle $C$ is \emph{flat} if $C$ bounds a disc $D\subseteq \Sigma$. Let $C_1,\ldots,C_n$ be flat cycles that bound discs $D_1,\ldots,D_n$, respectively. The cycles $(C_1,\ldots,C_n)$ are \emph{concentric} if $D_{i}\supseteq D_{i+1}$ for all $1\leq i < n$.

For positive integers~$r$, define a graph $H_r$ as follows.  Let $P_1, \dots, P_r$
be $r$ vertex disjoint (`horizontal') paths of length $r-1$, say $P_i = v_1^i\dots
v_r^i$. Let $V(H_r) = \bigcup_{i=1}^r V(P_i)$, and let
\begin{equation*}
\begin{split}
E(H_r) = \bigcup_{i=1}^r E(P_i) \cup
\Big\{&v_j^i v_j^{i+1} \mid \text{ $i,j$ odd};\ 1 \le i < r;\ 1 \le j \le r\Big\} \\
& \cup \Big\{v_j^i v_j^{i+1} \mid \text{ $i,j$ even};\ 1 \le i < r;\ 1 \le j \le r\Big\}.
\end{split}
\end{equation*}
The 6-cycles in $H_r$ are its \emph{bricks}. In the natural plane embedding of~$H_r$, these bound its `finite' faces. The outer cycle of the unique maximal 2-connected subgraph is called the \emph{boundary cycle}  of $H_r$.

Any subdivision $H = T H_r$ of $H_r$ will be called an \emph{$r$--wall}. The \emph{bricks} and the \emph{boundary cycle} of $H$ are its subgraphs that form subdivisions of the bricks and the boundary cycle of~$H_r$, respectively.    The \emph{first $n$ boundary cycles} $C_1,\dots,C_n$ of~$H_r$ are defined inductively: $C_n$~is the outer cycle (in the induced embedding) of the unique maximal 2-connected subgraph $H_r^{-(n-1)}$ of $H_r - (C_1\cup\ldots\cup C_{n-1})$.  An embedding of $H$ in a surface~$\Sigma$ is a \emph{flat} embedding, and $H$ is \emph{flat} in~$\Sigma$, if the boundary cycle $C$ of $H$ bounds a disc that contains a vertex of degree 3 of $H-C$.  We refer to the disc bounded by $C$
as $\Delta(\Sigma, H)$.

An $\alpha$--near embedding of a graph $G$ in some surface $\Sigma$ is {\it $\beta$--rich} if the following statements hold:
\begin{enumerate}%[(i)]
\item[$i.$]  $G_0'$ contains a flat $r$--wall $H$ for some $r \ge \beta$.
\item[$ii.$] For every vortex $V\in \V$ there are $\beta$ disjoint, concentric cycles $(C_1,\ldots,C_{\beta})$ in $G'_0$ that bound discs $(D_1,\ldots,D_\beta)$, respectively, the innermost disc $D_{\beta}$ contains $\Omega(V)$ and $H$ does not intersect with $D_1$.
\item[$iii.$] Every two vortices in $\V$ have distance at least $\beta$.
\item[$iv.$]  Let $V \in \V$ with $\Omega(V)=(w_1,\ldots,w_n)$. Then there is a linked decomposition of $V$ of adhesion at most $\alpha$ and a path $P$ in $V\cup \bigcup \W$ with $V(P\cap G_0)=\Omega(V)$, avoiding all the paths of the linkage of $V$, and traversing  $w_1,\ldots,w_n$ in their order.
\item[$v.$] For every vortex $V\in\V$ the society vertices $\Omega(V)$ are linked in $G'_0$ to the vertices of $H$ of degree 3 by a path system of $\beta$ disjoint paths and these paths have no inner vertices in $H$.
\item[$vi.$] All vortices in $\W$ are properly attached to $G_0$.
\end{enumerate}

\begin{Theorem}\label{richthm}
For every graph~$R$, there is an integer~$\alpha$ such that for every integer~$\beta$ there is an integer $w=w(R,\beta)$ such that the following holds.
Every graph~$G$ with~$tw(G)\geq w$ that does not contain $R$ as a minor has an $\alpha$--near, $\beta$--rich embedding in some surface~$\Sigma$ in which $R$ cannot be embedded.
\end{Theorem}

Here is an outline of how we shall use Theorem~\ref{richthm} in our proof of Theorem~\ref{thm:bigtw}. By Euler's formula, a graph embedded in a fixed surface has average degree at most $6+o(1)$ (in terms of its order).%
   \COMMENT{}
   The high connectivity we assumed for our graph~$G$ thus implies that, when we apply Theorem~\ref{richthm} to it, $G$ cannot be entirely embedded in~$\Sigma$: when the wall $H\subseteq G'_0$ gets large, the embedded subgraph $G'_0$ of $G$ must have many vertices of degree at most~$6$. These vertices send their remaining edges outside~$G'_0$: to the apex set~$A$, to components of $G_0 - G'_0$, or into the vortices~$G_1,\dots,G_{\alpha'}$.

Distinguishing vertices of large and small degree in~$G'_0$ will be crucial to our proof. However, we put the threshold a little higher than~6, at~$10p$. We shall first show, in Section~\ref{sec:subwall}, that by carefully choosing a subwall $H'$ of~$H$, we can ensure that the vertices of $G'_0$ in $\Delta(\Sigma,H')$ have large degree in~$G'_0$, and have no neighbours outside~$G'_0$ other than in~$A$. In Sections \ref{sec:linkages} and~\ref{sec:link} we then find a large linkage in $G'_0$ from a cycle deep inside $H'$ to vertices that have small degree in~$G'_0$. These vertices send many edges out of~$G'_0$. If these edges go directly to~$A$ or to components of $G_0 - G'_0$ (which  in turn sends many edges to~$A$, by the connectivity of~$G$), we can build from this linkage, some cycles in~$H'$ through which it passes, and many common neighbours in $A$ of the endvertices of our linkage or of those components, an instance of an $F(kp, k(p-3))$-minor which contains our desired $kK_p$-minor. Otherwise, most of the endvertices of our linkage send their many edges out of $G'_0$ into vortices, and many into the same vortex. We shall then find our $kK_p$ minor using that vortex (Section~\ref{lastsection}).

\section{Isolating a subwall in a disc with all degrees large}\label{sec:subwall}

Our aim in this section is to show that when we apply Theorem~\ref{richthm} to our highly connected graph~$G$, we can choose a subwall $H'$ of the wall $H$ so that the vertices of $G'_0$ in $\Delta(\Sigma,H)$ have large degree in~$G'_0$, and have no neighbours outside~$G'_0$ other than in~$A$.

\begin{Lemma}\label{lem:cleangrid}
Let $\alpha\in\N$ be as provided by Theorem~\ref{richthm} for $R = kK_p$. For every $r\in\N$ there exists $w\in\N$ such that every $(k(p-3)+14p+14)$-connected graph $G\not\Minor kK_p$ of tree-width at least $w$ admits an $\alpha$-near $\beta$-rich embedding for some $\beta \ge r$ such that there exists an r-wall $H'$ contained in $G_0'\cap \Delta(\Sigma, H)$ with the property that every vertex in $\Delta(\Sigma, H')$ has degree at least 10p in $G'_0$ and has no neighbour in $G-A$ outside $G'_0$.
\end{Lemma}

\begin{proof}
Let $r$ be given. We will choose $\beta= \beta(r)$ below; it must be sufficiently large to guarantee the $\beta$ wall $H$ in an $\alpha$-near $\beta$-rich embedding contains enough disjoint $r$-walls so that if none of these can serve as $H'$ for our lemma, we can combine them all to find a $kK_p$ minor.  Given such a $\beta$, the existence of $w$ is then implied by Theorem \ref{richthm}.  Let $G$ be a $(k(p-3)+14p+14)$-connected graph with an $\alpha$-near $\beta$-rich embedding in~$\Sigma$. Choose the $\alpha$-near embedding so that $|G'_0|$ is minimum. This implies that for every subwall $H'$ of~$H$ the graph $G'_0\cap\Delta(\Sigma,H')$ is connected: any component other than that containing~$H'$ could be included in $V_i$ for some $V \in \mathcal{W}$, decreasing~$|G'_0|$.

Consider a component $C$ of $G_0 - G'_0$, and pick a vertex $v\in C$. Then $C$ is separated from $G'_0$ in $G-A$ by the at most 3 vertices in $G'_0$.%
   \COMMENT{}
   Since $G$ is $(k(p-3)+14p+14)$-connected, this means that $C$ has at least $k(p-3)$ distinct neighbours in~$A$. Let $G'$ be obtained from $G$ by contracting every component $C$ of $G_0 - G'_0$ to one vertex; for every vertex $v\in C$ we denote this new vertex contracted from $C$ as~$v'$.%
   \COMMENT{}

Call a vertex $u$ of $G'_0$ in $\Delta(\Sigma,H)$ \emph{bad} if it has degree $< 10p$ in $G'_0$ or has a neighbour in $(G-A)-G'_0$. If $u$ has a neighbour $v$ in~$(G-A)-G'_0$, then $v$ must lie in~$G_0-G'_0$; recall that, by definition a $\beta$-rich $\alpha$-near embedding, the disc $\Delta(\Sigma,H)$ contains no vertex from any vortex $V \in \mathcal{V}$. In~$G'$, the contracted vertex $v'$ has $k(p-3)$ neighbours in~$A$. Similarly if $u$ has degree $<10$ in $G'_0$ but no neighbour in $(G-A)- G'_0$, then $u$ itself has more than $k(p-3)$ neighbours in~$A$, by the connectivity assumed for~$G$.

By making $\beta$ large enough in terms of $r$ and $\ell$ (see below), we can find in $H$ an instance of an $L(\ell)$-minor (an $\ell$-ladder) in which every branch set induces a subgraph in $H$ containing an $r$-wall, and these $r$-walls $H_i$ are sufficiently spaced out in $\Delta(\Sigma,H)$ that the discs $\Delta(\Sigma,H_i)$ are disjoint and not joined by edges of~$G'_0$. In particular, for any vortex $V \in \mathcal{W}$, the corresponding vertices $\Omega(V)$ meet at most one of these~$\Delta(\Sigma,H_i)$. If one of these discs $\Delta(\Sigma,H_i)$ contains no bad vertex, our lemma is proved with $H':= H_i$. So assume that each of them contains a bad vertex. Let $H_1,\dots, H_\ell$ be the $r$-walls from the branch sets of the `top' row of our $\ell$-ladder minor, and put $\Delta_i = \Delta(\Sigma,H_i)$ for $i=1,\dots,\ell$. For each~$i$, pick a bad vertex $u_i\in\Delta_i$. If $u_i$ has a neighbour $v_i$ in $(G-A)-G'_0$, its neighbour $v'_i$ in $G'$ has (in~$G$) at least $k(p-3)$ neighbours in~$A$, and these $v'_i$ are distinct for different~$i$. Let $G''$ be obtained from $G'$ by contracting the edge $u_i v'_i$, and call the contracted vertex~$w_i$. If $u_i$ has no neighbour in $(G-A) - G'_0$, then $u_i$ itself has $k(p-3)$ neighbours in~$A$; let us rename these $u_i$ as~$w_i$.

For each $i=1,\dots,\ell$, the vertex $w_i$ has, in~$G'$, a set $A_i$ of $k(p-3)$ neighbours in~$A$. We now choose $\ell$ large enough that for $kp$ values of~$i$, say those in~$I$, the sets $A_i$ coincide. (Notice that $\ell$ depends only on $\alpha$, $k$ and~$p$, all of which are constant.) Let $A'$ denote this common set $A_i$ for all $i\in I$. Together with $A'$ and the vertices $v'_i$ with $i\in I$, our instance of an $L(\ell)$-minor in $H'$ contains an instance of an $F(kp, k(p-3))$-minor in~$G'$: the $k(p-3)$ vertices in $A'$ form singleton branch sets, their neighbouring branch sets are sets $V(G'_0)\cap\Delta_i$ for $i\in I$, plus $v'_i$ as appropriate (recall that these sets are connected by the minimality of~$|G'_0|$), and the remaining branch sets found in our ladder~$L(\ell)$. Thus, $kK_p\minor F(kp, k(p-3))\minor G'\minor G$, contradicting our choice of~$G$.
\end{proof}

For easier reference later, let us summarize as a formal hypothesis the properties ensured by Lemma~\ref{lem:cleangrid} along with the aspects of a $\beta$-rich embedding we will need as we go forward.  We will be able to ensure these properties as long as the graph we are interested has sufficiently large tree width.  Let $\Sigma$ and $\alpha\in\N$ be as provided by Theorem~\ref{richthm} for $R = kK_p$ applied to the graph $G$ be a graph.  Let $r>0$ an integer.

\paragraph{\boldmath Hypothesis H$(G,r)$:} The graph $G$ is $(k(p-3)+14p+14)$-connected graph and has no $kK_p$ minor. The graph $G$ has an $\alpha$-near embedding satisfying the following properties:
\begin{itemize}
\item[$i.$] There is a flat $r$-wall $H$ in $G'_0$.
\item[$ii.$] Every vertex $v \in G_0' \cap \Delta(\Sigma, H)$ has degree at least $10p$ in $G_0' \cap \Delta(\Sigma, H)$ and for every vortex $V \in \mathcal{W}$, the vertices $\Omega(V)$ are disjoint from $G_0' \cap \Delta(\Sigma, H)$.
\item[$iii.$] Let $V \in \V$ with $\Omega(V)=(w_1,\ldots,w_n)$. Then, there is a linked decomposition $(X_1,\ldots,X_n)$ of $V$ of adhesion at most $\alpha$ and there is a path $P$ in $V\cup \bigcup \W$ with $V(P\cap G_0)=\Omega(V)$, the path $P$ is disjoint to all paths of the linkage of $V$ and traverses $w_1,\ldots,w_n$ in their linear order.
\item[$iv.$]All vortices in $\W$ are properly attached to $G_0$.
\end{itemize}%

\bigskip\noindent
   Lemma~\ref{lem:cleangrid} says that, for every $r\in\N$, every $(k(p-3)+14p+14)$-connected graph $G\not\Minor kK_p$ of large enough tree-width satisfies Hypothesis~H($G,r$). Note that if $G$ satisfies H($G,r$) then it also satisfies H($G,r'$) for every $r'\le r$: just take an $r'$-wall $H'$ inside the given $r$-wall~$H$.

%%%%%%%%%%%%%%%%%%%%%%%%%%%%%%%%%%%%%%%%%%%%%%%%%%

\section{Optimizing linkages}\label{sec:linkages}

In this section we prove three lemmas about linkages, which may also be of use elsewhere.

An $X$--$Y$ linkage $\P$ in a graph $G$ is \emph{singular}  if $V(\bigcup\P) = V(G)$ and $G$ does not contain any other $X$--$Y$ linkage.%
   \COMMENT{}

\begin{Lemma}\label{lem:singlink}
If a graph $G$ contains a singular linkage~$\P$, then $G$ has path-width at most~$|\P|$.
\end{Lemma}

\begin{proof}
Let $\P$ be a singular $X$--$Y$ linkage in~$G$. Applying induction on~$|G|$, we show that $G$ has a path-decomposition $(X_0,\dots,X_n)$ of width at most~$|\P|$ such that $X\subseteq X_0$. Suppose first that every $x\in X$ has a neighbour $y(x)$ in~$G$ that is not its neighbour on the path $P(x)\in\P$ containing~$x$. Then $y(x)\notin P(x)$ by the uniqueness of~$\P$.%
   \COMMENT{}
   The digraph on $\P$ obtained by joining for every $x\in X$ the `vertex' $P(x)$ to the `vertex' $P(y(x))$ contains a directed cycle~$D$.%
   \COMMENT{}
   Let us replace in $\P$ for each $x\in X$ with $P(x)\in D$ the path $P(x)$ by the $X$--$Y$ path that starts in~$x$, jumps to $y(x)$, and then continues along~$P(y(x))$. Since every `vertex' of~$D$ has in- and outdegree both~1 there, this yields an $X$--$Y$ linkage with the same endpoints as $\P$ but different from~$\P$. This contradicts our assumption that $\P$ is singular. Thus, there exists an $x\in X$ without any neighbours in $G$ other than (possibly) its neighbour on~$P(x)$. Consider this~$x$.

If $P(x)$ is trivial, then $x$ is isolated in $G$ and $x\in X\cap Y$. By induction, $G-x$ has a path-decomposition $(X_1,\dots,X_n)$ of width at most $|\P|-1$ with $X\setminus\{x\}\subseteq X_1$.%
   \COMMENT{}
   Add $X_0:= X$ to obtain the desired path-decomposition of~$G$. If $P(x)$ is not trivial, let $x'$ be its second vertex, and replace $x$ in $X$ by $x'$ to obtain~$X'$. By induction, $G-x$ has a path-decomposition $(X_1,\dots,X_n)$ of width at most $|\P|$ with $X'\subseteq X_1$. Add $X_0:= X\cup\{x'\}$ to obtain the desired path-decomposition of~$G$.
\end{proof}

Our next lemma will help us re-route segments of an $X$--$Y$ linkage~$\P$ in $G$ through a subgraph $H\subseteq G$, which may or may not intersect~$\bigcup\P$. Let $\Q$ be a set of disjoint paths that start in~$H$, have no further vertices in~$H$,%
   \COMMENT{}
   and end in~$\bigcup\P$. (They may have earlier vertices on~$\P$.) The \emph{$(\Q,H)$-segment} of a path $P\in\P$ is the unique maximal subpath of $P$ that starts and ends in a vertex of~$\bigcup\Q\cup H$; this subpath may be trivial, or even empty. We call $\Q$ an \emph{$H$--$\,\P$ comb} if the set of endvertices of $(\Q,H)$-segments of paths in~$\P$ equals the set of final vertices of paths in~$\Q$.

\begin{Lemma}\label{lem:comb}
Let $t$ be an integer, let $\P$ be an $X$--$\,Y$ linkage in a graph~$G$, and let $H\subseteq G$. If $G$ contains $t$ disjoint $H$--$\,(X\cup Y)$ paths, then $G$ contains an $H$--$\,\P$ comb consisting of at least $t$ paths.
\end{Lemma}

\begin{proof}
Let $\Q$ be a set of as many disjoint $H$--$\,(X\cup Y)$ paths as possible, chosen with the least possible number of edges not in~$\bigcup\P$. By the maximality of~$\Q$, every endvertex $v$ of a $(\Q,H)$-segment of a path $P\in\P$ lies on a path $Q\in\Q$.%
   \COMMENT{}
   By our choice of~$\Q$, the final segment $vQ$ of $Q$ then lies in~$P$. Deleting the final segments $\interior v Q$ after $v$ for each such endvertex of a $(\Q,H)$-segment turns $\Q$ into an $H$--$\,\P$ comb.%
   \COMMENT{}
\end{proof}

While it is not typically true that a subset of a comb will again be a comb, the following is true.  We omit the straightforward proof.
\begin{Observation}\label{obs:subcomb}
Let $\P$ be a linkage and $H$ a subgraph in a graph $G$.  Let $\mathcal{R}$ be an $H - \P$ comb.  Then for any sublinkage $\P'$ of $\P$, the linkage $$\mathcal{R}' : = \{R \in \mathcal{R} : \text{ there exists a $(\mathcal{R}, H)$-segment in $\P'$ sharing an endpoint with $R$}\}$$ is a $H- \P'$ comb.
\end{Observation}

We finally turn to linkages in graphs that are, for the most part, embedded in a cylinder.
Let $C_1, \dots, C_s$ be disjoint cycles.  A linkage $\P$ is \emph{orthogonal} to $C_1,\dots, C_s$ if for all $P \in \mathcal{P}$, $V(P) \cap V(C_i)\neq \emptyset$ for all $1 \le i \le s$ and $P$ intersects the cycles $C_1, C_2, \dots, C_s$ in that order when traversing $P$ from one endpoint to the other.  Moreover, each of the graphs $P\cap C_i$ is a path (possibly consisting of a single vertex).
      The next lemma is a weaker version of Theorem 10.1 of~\cite{JORG2}.  We include its proof for completeness.

\begin{Lemma}\label{lem:ortholink}
Let $s$, $s'$, and $t$ be positive integers with $s\ge  s' + t $. Let $G'$ be a graph embedded in the plane and let $(C_1, \dots, C_s)$ be concentric cycles in $G'$.
Let $G''$ be another graph, with $V(G')\cap V(G'') \subseteq V(C_1)$. Assume that $G'\cup G''$ contains an $X$--$Y$ linkage $\P = \{P_1, \dots, P_t\}$  with $ X \subseteq C_s$ and $Y \subseteq C_1$.  Then there exist concentric cycles $(C_1', \dots, C'_{s'})$ in $G'$, a set $X' \subseteq V(C_s')$, and an $X'$--$Y$ linkage $\mathcal{P}'$ in $G'\cup G''$ such that $\mathcal{P}'$ is orthogonal to $C_1', \dots, C_{s'}'$.
\end{Lemma}

\begin{proof}
Assume the lemma is false, and let $G'$, $G''$, $\mathcal{P}$, and $(C_1, \dots, C_s)$ form a counterexample containing a minimal number of edges.  To simplify the notation, we let $G = G' \cup G''$.
By minimality, it follows that the graph
$G = \bigcup_1^s C_i \cup \mathcal{P}$.  Also, for all $P \in \mathcal{P}$ and for all $1 \le i \le s$, every component of $P \cap C_i$ is a single vertex.  If $P \cap C_i$ had a component that was a non-trivial
path containing an edge $e$, then $G'/e$ would form a counterexample with fewer edges.
Similarly, we conclude that $V(G) = V(\mathcal{P})$.

Note that no subpath $Q \subseteq \mathcal{P} \cap G'$ that is internally disjoint from $\bigcup_1^s C_i$ has
both endpoints contained in $C_j$ for some $1 \le j \le s$.  There are two cases to consider.  If $Q \subseteq \Delta(C_s)$, we violate our choice of a minimal counterexample by restricting the $\P$ path containing $Q$ to a subpath from $Y$ to $V(C_s)$ avoiding the edges of $Q$.  If $Q \nsubseteq \Delta(C_s)$, we could reroute $C_j$ through the path $Q$ to find $s$ concentric cycles in $G'$ and again contradict our choice of a counterexample containing a minimal number of edges.
We claim:
\begin{equation}\label{cl:2}
\emtext{
The graph $G$ consists of a singular linkage.
}
\end{equation}
To see that the claim is true, observe that $E(\mathcal{P})$ is disjoint
from $E\left ( \bigcup_1^s C_i\right)$.  It follows that
if there exists a linkage $\overline{\mathcal{P}}$ from $X$ to $Y$ distinct
from $\mathcal{P}$, then at least one of the edges of $\mathcal{P}$ is not contained in
$\overline{\mathcal{P}}$.  We conclude that the subgraph $\bigcup_1^s C_i \cup \overline{\mathcal{P}}$
forms a counterexample to the claim with fewer edges, a contradiction.  This proves $(\ref{cl:2})$.

A \emph{local peak} of the linkage $\mathcal{P}$ is a subpath $Q \subseteq \mathcal{P}$
such that $Q$ has both endpoints on $C_j$ for some $j>1$ and every internal vertex
of $Q \cap \left ( \bigcup_{i \neq j} V(C_i) \right) \subseteq V(C_{j-1})$.  As we have seen
above, it must then be the case that $V(Q) \cap V(C_{j-1}) \neq \emptyset$ when $j>1$.

We claim the following.
\begin{equation}\label{cl:3}
\emtext{
For all $j>1$, there does not exist a local peak with endpoints in $C_j$.
}
\end{equation}
Fix $Q$ to be a local peak with endpoints in $C_j$ with $Q$ chosen over
all such local peaks so that
$j$ is maximal.  Assume $Q$ is a subpath of $P \in \mathcal{P}$.
Let the endpoints of $Q$ be $x$ and $y$.  Lest we re-route $P$ through $C_j$
and find a counter-example containing fewer edges, there exists
a component $P' \in \mathcal{P}$ intersecting the subpath of $C_j $
linking $x$ and $y$.  By planarity, $P'$ either contains a subpath internally disjoint from the union of the $C_i$ with both endpoints in $C_s$, or $P'$ contains a subpath forming
a local peak with endpoints in $C_{j-1}$. Either is a contradiction to our choice of a minimal counterexample.  This proves $(\ref{cl:3})$.

An immediate consequence of $(\ref{cl:2})$ and $(\ref{cl:3})$ is the following.
For every $P \in \mathcal{P}$, let $x$ be the endpoint of $P$ in $X$ and
let $y$ be the vertex of $V(C_1) \cap V(P)$ closest to $x$ on $P$.
Define the path $\overline{P}$ be the subpath $xPy$ of $P$.
The path $\overline{P}$ is orthogonal to the cycles $C_1, \dots, C_s$.
In fact, $\overline{P} \cap C_i$ is a single vertex for each $1 \le i \le s$.  The final claim will complete
the proof.
\begin{equation}\label{cl:4}
\emtext{
For all $P \in \mathcal{P}$, the path $P - \overline{P}$ does not intersect $C_{t+1}$.
}
\end{equation}
To see $(\ref{cl:4})$ is true, fix $P \in \mathcal{P}$ such that
 $(P - \overline{P}) \cap C_{t+1} \neq \emptyset$.
It follows now from $(\ref{cl:3})$
 that $P - \overline{P}$ contains a subpath $Q$ with one endpoint in $C_{t+1 }$
and one endpoint in $C_1$ such that $Q$ is orthogonal to the cycles $C_{t+1}, C_{t}\dots, C_{1}$.
By the planarity of $G'$,
we see that $G$
contains a subgraph isomorphic to the subdivision of the $(t+1)\times (t+1)$ grid.  This
contradicts $(\ref{cl:2})$ and Lemma \ref{lem:singlink}, proving $(\ref{cl:4})$.

We conclude that $\mathcal{P}$ is orthogonal to the $s'$ disjoint cycles $C_{s}, C_{s-1}, \dots, C_{t+1}$.  This contradicts our choice of $G$, and the lemma is proven.
\end{proof}

%%%%%%%%%%%%%%%%%%%%%%%%%%%%%%%%%%%%%%%%%%%%%%%

\section{Linking the wall to a vortex}\label{sec:link}

Consider a graph $G$ satisfying Hypothesis H($G,r$). Our first aim in this section is to find a large linkage from a cycle deep inside~$H$ to vertices of small degree in~$G'_0$. By Lemma~\ref{lem:ortholink} we shall be able to assume that this linkage is orthogonal to a pair of cycles $C$ and $C'$. If the many of the last vertices of our linkage send many edges to~$A$, or an edge to a component of $G_0-G'_0$ (which in turn sends many edges to~$A$, by the connectivity of~$G$), we shall be able to convert the cycles $C$ and $C'$, the linkage, and those neighbours into an $F(kp, k(p-3))$-minor, completing the proof. If not, then most of those last vertices send many edges into vortices. As we have only a bounded number of vortices, many send their edges to the same vortex. That case we shall treat in Section~\ref{lastsection}.

\begin{Lemma}\label{lem:biglink}
For all positive integers $t$ and $s$ there exists an integer $R = R(s,t)$ such for every graph $G$ satisfying Hypothesis ${\rm H}(G,r)$ with $r \ge R$ there are $t$ disjoint $X$--$\,Y$ paths in~$G'_0$, where $X$ is the vertex set of the $s$'th boundary cycle $C_s$ of~$H$, and $Y:= \{v \in V(G'_0): \deg_{G'_0}(v) < 10p\}$.%
   \COMMENT{}
   \end{Lemma}

\begin{proof}
If the desired paths do not exist then, by Menger's theorem, $G'_0$~has a separation $(A, B)$ of order less than $t$ with $X \subseteq A$ and
$Y\subseteq B$. By the choice of~$Y$, every vertex in $A\setminus B$ has degree at least $10p$ in~$G'_0$. The sum of all these degrees is at least $10p\, |A\setminus B|$, so $G'_0 [A]$ has at least $5p\, |A\setminus B|$ edges. As $|A|\ge |X|\ge r-4s$, and $\Sigma$ is determined by our constants $p$ and~$k$, choosing $R$ sufficiently large in terms of $s$ and~$t$ yields%
   \COMMENT{}
 $$ 5p\, |A\setminus B|\ge 5p\,(|A| - t) > 3\,|A| - 3\,\chi(\Sigma),$$
 which is the maximum number of edges a graph of order $|A|$ embedded in $\Sigma$ can have (by Euler's formula). As $G'_0[A]$ is such a graph, this is a contradiction. %
   \COMMENT{}
\end{proof}

Our next lemma says that by rerouting the paths if necessary we can make the linkage from Lemma~\ref{lem:biglink} orthogonal to two concentric cycles. Recall that the wall $H$ in Hypothesis~${\rm H}(G,r)$ is flat; we think of the topological disc $\Delta(\Sigma,H)\subseteq\Sigma$, which contains $H$ and is bounded by its outer cycle~$C_1$, as a disc in~$\R^2$.%
   \COMMENT{}

\begin{Lemma}\label{lem:2cyc}
Let $t$ be an integer. Let $G$ be a graph satisfying Hypothesis ${\rm H}(G,r)$ for some~$r$ large enough that $H$ has boundary cycles $C_1,\dots,C_{t+2}$. Suppose further that $G'_0$ contains an $X$--$\,Y$ linkage $\P$ of order~$t$, where $X\subseteq V(C_{t+2})$ and $Y\subseteq V(G'_0)\setminus \Delta(\Sigma,H)$.   Then $\bigcup\P\cup C_1\cup\ldots\cup C_{t+2}\subseteq G'_0$ contains disjoint cycles $C'_1, C'_2$ in $G_0' \cap \Delta(\Sigma, H)$, and an $X'$--$Y$ linkage%
   \COMMENT{}
   orthogonal to~$C'_1,C'_2$ with $X'\subseteq V(C_{t+2})$.
\end{Lemma}

\begin{proof}
   Let $Y'$ be the set of the last vertices in $\Delta(C_1)$ of paths in~$\P$; this is a subset of~$V(C_1)$. Let $\P'$ be the set of $X$--$Y'$ paths contained in the paths in~$\P$ (one in each).
Let $G'$ be the union of all the cycles $C_1,\dots,C_{t+2}$ and the subpaths in~$\Delta(C_1)$ of paths in~$\P'$. Then $G'$ is planar and $(C_1, \dots, C_{t+2})$ form a set of concentric cycles. Let $G''$ be the union of the remaining segments of paths in~$\P'$; then $G'\cup G'' = \bigcup\P'\cup C_1\cup\ldots\cup C_{t+2}\,$, and $V(G')\cap V(G'')\subseteq V(C_1)$. Applying Lemma~\ref{lem:ortholink} to the linkage $\P'$ in $G'\cup G''$, we obtain a two disjoint cycles $C'_1$ and $C'_2$ contained in $\Delta(C_1)$ and a set $X' \subseteq(C_2')$ such that the cycles are orthogonal to an $X'$--$Y'$ linkage~$\P''$ in $G'\cup G''$. Append to this linkage the $Y'$--$Y$ paths contained in the paths from~$\P$ (which meet $G'\cup G''$ only in~$Y'$, by the choice of~$Y'$) to obtain the desired linkage for the lemma.
\end{proof}

\noindent
   Note that the linkage obtained in Lemma \ref{lem:2cyc} has the same order as~$\P$, since the target set $Y$ remained unchanged. The proof of the lemma could clearly be modified to provide
a set of $s$ cycles orthogonal to the linkage, for arbitrary~$s$, rather than just two, but we shall only need two in the following arguments.

\bigbreak

We return now to the linkage provided by Lemma \ref{lem:biglink}. Using Lemma~\ref{lem:2cyc}, we show that all but a bounded number of the paths of that linkage lie in~$G$ (that is, contain no virtual edges) and end in vortices.

\begin{Lemma}\label{lem:2}
Let $t\ge 3kp\binom{\alpha}{k(p-3)}$, and let $G$ be a graph satisfying Hypothesis ${\rm H}(G,r)$, for some $r$ large enough that the $(t+2)$th boundary cycle $C_{t+2}$ of~$H$ exists. Then $G'_0$ contains no $X$--$\,(Y\cup Z)$ linkage of order~$t$ such that $X\subseteq V(C_{t+2})$, the set $Z$ contains vertices of $\bigcup_{V \in \W} \Omega(V)$, and $Y\subseteq \{v \in V(G'_0): \deg_{G'_0}(v) < 10p\}\setminus \bigcup_{V \in \V \cup \W} \Omega(V)$.
\end{Lemma}

\begin{proof}
Suppose there is an $X$--$\,(Y\cup Z)$ linkage in $G'_0$ as stated. Since $Y\cup Z\subseteq V(G'_0)\setminus \Delta(\Sigma,H)$,%
   \COMMENT{}
   Lemma~\ref{lem:2cyc} provides us with an $X'$--$\,(Y\cup Z)$ linkage $\P$ in $G'_0$ that is orthogonal to two cycles $C'_1,C'_2$ contained in $\Delta(\Sigma,H)$, where again $X'\subseteq V(C_{t+2})$. Since $\P$ has the same target set $Y\cup Z$ as the original linkage, it also has the same order $t\ge 3kp\binom{\alpha}{k(p-3)}$.

Each path in $\P$ contains at most one vertex contained in a vortex, since this will be its last vertex.  Also, for every vortex $V \in \W$, $\Omega(V)$ intersects at most 3 paths in $\P$.
We find a subset $\P'$ of $\P$ of order $kp\binom{\alpha}{k(p-3)}$ and for every $P \in \P'$ such that $P$ intersects some vortex, we assign a vortex $V(P)$ such that for $P, Q \in \P'$, $V(P) \neq V(Q)$.  We can construct the subset $\P'$ and the vortex assignments greedily - we begin considering $\P$, and as long some path $P$ intersects an unassigned vortex $V$, we set $V(P):=V$ and delete any other path $Q$ with both $Q \cap \Omega(V) \neq \emptyset$ and $V(Q)$ undefined.

By definition of~$Y$ and the connectivity of~$G$, every last vertex $y\in Y\setminus Z$ of a path $P\in \P'$ has a set $A_P$ of $k(p-3)$ distinct neighbours in~$A$. By definition of~$Z$, every last vertex $z\in Z$ of a path $P\in\P'$ sends an edge to some component $C$ of $V(P)-G'_0$. In~$G-A$, a set of at most three vertices of $G'_0$ (which includes~$z$) separates $C$ from the rest of~$G'_0$. Hence by the connectivity of~$G$, the component $C$ has a set $A_P$ of $k(p-3)$ distinct neighbours in~$A$. For every such~$z$, contract the component $C$ on to~$z$. (By definition of~$\P'$ and the assignment $V(P)$, these $C$ are distinct, and hence disjoint, for different~$z$.) In the resulting minor $G'$ of~$G$, the vertex $z$ is adjacent to every vertex in~$A_P$ (for the $P\in\P'$ ending in~$z$).

Since $|\P'| = kp\binom{\alpha}{k(p-3)}$, there is a subset $\P''$ of~$\P'$ of order~$kp$ such that for all the paths $P\in\P''$ their sets $A_P$ conincide;%
   \COMMENT{}
   let us write $A'$ for this subset of $A$ of order~$k(p-3)$.

Of each path $P\in\P''$ let us keep only its segment $P'$ between $C'_2$ and~$C'_1$, contracting the final segment of $P$ that follows its vertex $v_P$ in $C'_1$ on to~$v_P$. In the minor $G''$ of~$G'$ obtained by all these contractions, the final vertices $v_P$ of the paths $P'$ with $P\in\P''$ are adjacent to all the $k(p-3)$ vertices in~$A'$. The cycles $C'_1$ and~$C'_2$, the $kp$ paths~$P'$ with $P\in\P''$, and the edges between the vertices $v_P$ and $A'$ together contain%
   \COMMENT{}
   a subdivided fan $F(kp, k(p-3))$. Thus,
 $$kK_p\minor F(kp, k(p-3))\minor G''\minor G'\minor G\,,$$
 a contradiction.
\end{proof}

Since $b\le\alpha$, Lemma~\ref{lem:2} implies that of paths from the linkage of Lemma \ref{lem:biglink} some unbounded number end in the same~$W_i$. These have small degree in~$G'_0$

%%%%%%%%%%%%%%%%%%%%%%%%%%%%%%%%%%%%%%%%%%%%%%%%

\section{Proof of Theorem \ref{thm:bigtw}}\label{lastsection}

Let $r$ be the integer $R(s,t)$ provided by Lemma~\ref{lem:biglink} for
$$t = 2\alpha\left ( kp\binom{2\alpha}{k(p-3)} + k\left (\binom{p}{2} + 1\right)\binom{\alpha}{ p} \right) + 3kp\binom{\alpha}{k(p-3)}$$
and $s = t+2$. Let $w$ be large enough that, by Lemma~\ref{lem:cleangrid}, every ($k(p-3)+14p+14$)-connected graph $G\not\Minor kK_p$ of tree-width at least~$w$ contains an $r$-wall $H$ such that $(G,H)$ satisfies Hypothesis~${\rm H}(G,r)$, for this~$r$.

For our proof of Theorem \ref{thm:bigtw}, let $G$ be a ($k(p-3)+14p+14$)-connected graph of tree-width at least~$w$; we have to show that $G\Minor kK_p$. Suppose not. Then $(G,H)$ satisfies Hypothesis~${\rm H}(G,r)$ for the value of $r$ defined above, by our choice of~$w$.

Let $C_1, C_2, \dots, C_{t+2}$ be the first $t+2$ boundary cycles of~$H$. By Lemma~\ref{lem:biglink}, there are $t$ disjoint paths in~$G'_0$ from $V(C_{t+2})$ to vertices of degree $<10p$ in~$G'_0$. By Lemma~\ref{lem:2}, all but at most $3kp\binom{\alpha}{k(p-3)}$ of these paths intersect exactly one vortex $V$ at their endpoints, and furthermore, this vortex $V$ is among the $\alpha'$ vortices of $\V$.  \COMMENT{}
 Since $\alpha'\le\alpha$, at least $1/\alpha$ of these paths end in the same vortex, say $V_a = (G_a, \Omega_a)$. These paths, then, form an $X$--$Y$ linkage $\P$ in~$G$ (i.e. the linkage does not contain any of the virtual edges of $G_0'$) of order
\begin{equation}\label{sizeP}
 |\P|\ge 2\left ( kp\binom{2\alpha}{k(p-3)} + k\left (\binom{p}{2} + 1\right)\binom{\alpha}{ p} \right),
\end{equation}
 with $X\subseteq V(C_{t+2})$ and $Y\subseteq \Omega_a =: \{w_1,\dots,w_m\}$.

Let us add to the graph~$G_a$ all the vertices from~$A$ (together with the edges they send to~$G_a$), putting them in every part of its vortex decomposition. This does not affect our assumption that this decomposition is linked, since every vertex in $A$ becomes a trivial path in the linkage through~$G_a$.
The new (induced) subgraph $G_a$ of $G$ has a path decomposition $(U_1,\dots,U_m)$ with the following properties (where $U^+_i := U_i\cap U_{i+1}=: U^-_{i+1})\,$:
\begin{itemize}
\item $A\subseteq U_i$ for all $i=1,\dots,m$;
\item $U_i\cap \Omega_a = \{w_{i-1},w_i\}$ for all $i=1,\dots,m$ with $w_0:=w_1$;
\item all the sets $U^+_i$ and $U^-_i$ have the same order~($\le 2\alpha$);
\item $G_a-\Omega_a$ contains a $(U^+_1\setminus\{w_1\})$--$(U^-_m\setminus\{w_{m-1}\})$ linkage~$\Q$.
\end{itemize}

For each $i=0,\dots,m$,%
   \COMMENT{}
   let $H_i = G[U_i\cup U_{i+1}]$ (putting $U_0 = \{w_1\}$ and $U_{m+1} =\{w_m\}$).%
   \COMMENT{}
   The set $U_i^-\cup U_{i+1}^+\cup\{w_i\}$ of size at most~$4\alpha+1$ separates $H_i$ from the rest of~$G$ (put $U_0^- = U_{m+1}^+ = \emptyset$). Let $\Q_i$ be the set of the segments in~$H_i$ of paths in~$\Q$. These are $U^-_i$--$\,U^+_{i+1}$ paths, one for each $Q\in\Q$, when $1<i<m$. We write $\T_i$ for the set of trivial paths in~$\Q_i$; when $1<i<m$, this is the set
 $$\T_i = \big\{\{v\}\mid v\in U^-_i\cap U^+_{i+1}\big\}\subseteq \Q_i\,.$$
 Note that $\T_i$ contains every path~$\{v\}$ with $v\in A$, and that $|\bigcup\T_i|\le |\Q| < 2\alpha$.

Deleting at most half the paths in~$\P$, we can ensure that for the remaining linkage $\P'\subseteq \P$ there is no $i<m$ such that both $w_i$ and $w_{i+1}$ are endpoints of a path in~$\P'$. Let ${I_1}\subset \{1,\dots,m\}$ be the set of those $i$ for which $w_i$ is the final vertex of a path in~$\P'$.

For each~$i\in {I_1}$, let $J_i$ denote the component of $H_i - U_i^- - U_{i+1}^+$ containing~$w_i$. Note that $J_i\cap\bigcup\T_i = \emptyset$ for each~$i$,%
   \COMMENT{}
   and that the $J_i$ are disjoint for different~$i\in {I_1}$. Let ${I_2}\subseteq {I_1}$ be the set of those $i\in {I_1}$ for which $J_i$ has at least $k(p-3)$ neighbours in~$\bigcup\T_i$, and put ${I_3}:= {I_1}\setminus {I_2}$. Let us show that
\begin{equation}\label{sizeI}
 |I_3|\ge k\left (\binom{p}{2} + 1\right)\binom{\alpha}{ p}.
\end{equation}

Suppose not; then $|{I_2}|\ge kp\binom{2\alpha}{k(p-3)}$, by~\eqref{sizeP}. For each $i\in {I_2}$, the at least $k(p-3)$ neighbours of $J_i$ in $\bigcup\T_i$ lie on different paths in~$\Q$. Since $|\Q|\le 2\alpha$, there is a set of $k(p-3)$ paths $Q$ in~$\Q$ and a set $I\subseteq {I_2}$ of order $kp$ such that for each of those $Q$ and every $i\in I$ we have $Q\cap H_i\in\T_i$ and the unique vertex in this graph sends an edge to~$J_i$. Contract each of these~$Q$ to one vertex, and contract each $J_i$ with $i\in I$ on to its vertex~$w_i$. Then each of these $kp$ vertices $w_i$ is adjacent to those $k(p-3)$ vertices contracted from paths in~$\Q$. Together with the $kp$ paths in $\P$ ending in these~$w_i$ and the cycles $C_1,\dots,C_{t+2}$ in our wall~$H$, we obtain a fan $F(kp, k(p-3))$ as in the proof of Lemma~\ref{lem:2}, contradicting our assumption that $G\not\Minor kK_p$. This proves~\eqref{sizeI}.

Let $\P''$ be the set of paths in $\P'$ ending in some $w_i$ with $i\in {I_3}$. For every $i\in {I_3}$, the graph $J_i$ has at most  $k(p-3)-1$ neighbours in~$\bigcup\T_i$. Our plan now is to find some fixed paths $Q^1,\dots, Q^p\in\Q$ and many indices $i\in {I_3}$, one for every edge in~$kK_p$, such that for each of these $i$ the segments $Q_i^j := Q^j\cap H_i$ are non-trivial%
   \COMMENT{}
   and we can connect two of them by a path through~$J_i$. (This will require some re-routing of $\Q_i$ inside~$H_i$.) Dividing the linkage $(Q^1,\dots, Q^p)$ into $k$ chunks kept well apart by the $k-1$ subgraphs~$H_i$ between them (in which all these paths have non-trivial segments; it is here only that we need the non-triviality of segments), and contracting the $p$ paths in each chunk to $p$ vertices, we shall thus obtain our desired $kK_p$ minor.

Let us begin by choosing the segments $Q_i^1,\dots, Q_i^p$ locally for each~$i\in {I_3}$, allowing the choice of $Q^1,\dots, Q^p$ to depend on~$i$. It will be easy later to find enough $i$ for which these choices agree. Let us prove the following:
\begin{equation}\label{reroute}
\begin{split}
&\emtext{For every $i\in {I_3}$ there are paths $Q^1,\dots,Q^p\in\Q$ with $Q_i^1,\dots,Q_i^p\in\Q_i\setminus\T_i$%
   \COMMENT{}
   such}\\
&\emtext{that for every choice of $1\le j < \ell \le p$ there is a linkage $(\hat Q_i^1,\dots,\hat Q_i^p)$ in $H_i$}\\
&\emtext{equivalent to $(Q_i^1,\dots,Q_i^p)$ for which $J_i - (\hat Q_i^1\cup\ldots\cup\hat Q_i^p)$ contains a path $R_i^{j,\ell}$}\\
&\emtext{from a vertex adjacent to~$\hat Q_i^j$ to a vertex adjacent to~$\hat Q_i^\ell$.}
\end{split}
\end{equation}

To prove~\eqref{reroute}, let $i\in {I_3}$ be given. Note that if any vertex $v$ of $J_i$ sends $p+1$ edges to $U_i^-\setminus\left(\{w_{i-1}\}\cup \T_i\right)$ or to~$U_{i+1}^+\setminus\left( \{w_{i+1}\}\cup \T_i\right)$, the proof of~\eqref{reroute} is immediate with~$R_i^{j,\ell} = \{v\}\,$: since $v$ lies on at most one of the $p+1$ non-trivial paths in $\Q$ to which it sends an edge, we can find $p$ such paths avoiding~$v$, no re-routing being necessary.%
   \COMMENT{}
So let us assume that this is not the case.

Consider the graph $J_i-w_i$. As $i\in I_3$, the vertex $w_i$ has fewer than $k(p-3)$
neighbours in $\bigcup\T_i$, fewer than $10p$ neighbours in~$G'_0$ (by definition of~$\P$), and at most $2p$ neighbours  in $(U_i^-\cup U_{i+1}^+) \setminus \left( \{w_{i-1}, w_{i+1}\} \cup \T_i\right)$. As $w_i$
has degree at least $k(p-3) + 14p+14$ in $G$, the graph $J_i - w_i$ is non-empty. By the same argument,%
   \COMMENT{}
\begin{equation}\label{density}
\delta(J_i - w_i) \ge \Big(k(p-3) + 14p+14\Big) - \Big(k(p-3)-1\Big) -  2p - 3 = 12(p+1)\,.
\end{equation}
 By Mader's theorem~\cite[Thm.~1.4.3]{DiestelBook05} and the main result from~\cite{thomas} (which says that $2s$-connected graphs of average degree at least $10s$ are $s$-linked), \eqref{density}~implies that $J_i-w_i$ has a ${(p+1)}$-linked subgraph~$H'_i$. In particular, $|H'_i|\ge 2p+2$. Let $Z_i$ consist of the vertices $w_{i-1}, w_i, w_{i+1}$ and the neighbours of $J_i$ in~$\bigcup\T_i$. As $i\in {I_3}$ we have $|Z_i|\le k(p-3)+2$, so $G - Z_i$ is still $2p$-connected. Since $H'_i\subseteq J_i-w_i$, the graph $H'_i$ has no vertex in~$Z_i$. By Menger's theorem, there are $2p$ disjoint paths in $G - Z_i$ from $H'_i$ to our wall~$H$. By definition of~$J_i$, their first vertices outside~$J_i$ lie in $U_i^-\cup U_{i+1}^+$ (recall that this set and $w_i$ together separate $H_i$ from~$H$ in~$G$),%
   \COMMENT{}
   and hence on a path in~$\Q_i\setminus\T_i$.%
   \COMMENT{}
   By Lemma~\ref{lem:comb}, there exists an $H'_i$--$(\Q_i\setminus\T_i)$ comb of at least $2p$ paths. Each path in $\Q$ meets at most two of them. Observation \ref{obs:subcomb} implies that we can find $p$ paths $Q^1,\dots,Q^p\in\Q$ such that $Q_i^1,\dots,Q_i^p\in\Q_i\setminus\T_i$ (as in~\eqref{reroute}) together with an $H'_i$--$\{Q_i^1,\dots,Q_i^p\}$ subcomb~$\RR$ meeting all of $Q^1,\dots,Q^p$. Let $\bar Q_i^q$ denote the $(\RR,H'_i)$-segment of $Q_i^q$, for each $q=1,\dots,p$; these segments are non-empty,%
   \COMMENT{}
   but they may be trivial.

We now define the paths $\hat Q_i^1,\dots,\hat Q_i^p$. For all $q$ whose $\bar Q_i^q$ is trivial we let $\hat Q_i^q = Q_i^q$. For those $q$ whose $\bar Q_i^q$ is non-trivial, we let $h_1^q\in H'_i$ be the starting vertex of the path $R_1^q\in\RR$ that ends on the first vertex of~$\bar Q_i^q$, and let $h_2^q\in H'_i$ be the starting vertex of the path $R_2^q\in\RR$ that ends on the last vertex of~$\bar Q_i^q$. Our aim is to link $h_1^q$ to $h_2^q$ in~$H'_i$ for each~$q$, but we must define $R_i^{j,\ell}$ at the same time. If $\bar Q_i^j$ is trivial, let $r^j\in H'_i$ be the starting vertex of the unique path $R^j\in\RR$ that ends on~$\bar Q_i^j$. If $\bar Q_i^j$ is non-trivial, let $r^j\in H'_i$ be a neighbour of $h_1^j$ in~$H'_i - \bigcup\RR$; such a neighbour exists, since~$H'_i$, being $(p+1)$-linked, is $(2p+1)$-connected~\cite[Ex.~3.22]{DiestelBook05}.%
   \COMMENT{}
   Define $r^\ell$ analogously.%
   \COMMENT{}
   Now choose a linkage in $H'_i$ consisting of a path $R = r^j\dots r^\ell$ and paths $R^q = h_1^q\dots h_2^q$ for all those $q$ such that $\bar Q_i^q$ is non-trivial. For these~$q$, let $\hat Q_i^q$ be obtained from $Q_i^q$ by replacing $\bar Q_i^q$ with $R_1^q\cup R^q\cup R_2^q$. If both $\bar Q_i^j$ and $\bar Q_i^\ell$ are trivial, let $R_i^{j,\ell}$ be the interior of the path $R^j\cup R\cup R^\ell$. If $\bar Q_i^j$ is trivial but $\bar Q_i^\ell$ is not, let $R_i^{j,\ell}$ be the path $R^j\cup R$ minus its first vertex. If $\bar Q_i^\ell$ is trivial but $\bar Q_i^j$ is not, let $R_i^{j,\ell}$ be the path $R\cup R^\ell$ minus its last vertex. If neither $\bar Q_i^j$ nor $\bar Q_i^\ell$ is trivial, let $R_i^{j,\ell}$ be the path~$R$. This completes the proof of~\eqref{reroute}.

By~\eqref{sizeI}, we can find a set ${I_4}\subseteq {I_3}$ of $k\left({p\choose2}+1\right)$ indices~$i$ in~${I_3}$ for which the choice of paths $Q^1,\dots,Q^p$ in~\eqref{reroute} coincides. (Recall that these paths are always chosen from the original vortex linkage of order~$\le\alpha$, since the trivial paths $\{v\}$ with $v\in A$ which we added later lie in every~$\T_i$.) For notational reasons only, let $\hat p := {p\choose2}$. Divide ${I_4}$ into $k$ segments
 $$(i_1^1,\dots,i_{\hat p}^1,i^1)\,,\ \dots\,,\ (i_1^k,\dots,i_{\hat p}^k,i^k)$$
 of length ${p\choose2}+1$. For every upper index $n=1,\dots,k$ contract in each of $Q^1,\dots,Q^p$ the segment from $H_{i_1^n}$ to~$H_{i_{\hat p}^n}$ (inclusive) to a vertex, and make these vertices into a $K_p$ minor using the paths $R_i^{j,\ell}$ from~\eqref{reroute} for subdivided edges,%
   \COMMENT{}
   one for each $i=i_1^n,\dots,i_{\hat p}^n$. Note that the $k$ instances of a $K_p$ minor thus obtained are disjoint, because they are `buffered' by the unused segments of the paths $Q^1,\dots,Q^p$ in $H_{i^n}$ for $n=1,\dots,k-1$.

%%%%%%%%%%%%%%%%%%%%%%%%%%%%%%%%%%%%%%%%%%%%%

\section{Tightness of the connectivity bound}\label{sec:tight}

The goal of this section will be to provide a construction of a graph $G_{n, k, p}$ for
all integers $p \ge 5$, $k \ge p$, and $n \ge 1$, such that the graph $G_{n, k, p}$
does not contain $k$ disjoint instances of $K_p$ as a minor, nor does the
graph $G_{n, k, p}$ contain a subset $X$ of vertices with $|X| \le n$ such that
$G - X$ does not contain $K_p$ as a minor. Moreover, we will construct
such a graph $G_{n, k, p}$ that is $(k(p-3) - \frac{(p-3)(p-4)}{2}-6)$-connected.
This will imply that the connectivity bound obtained in Theorem \ref{thm:main} is
best possible for all fixed $p$, $p \ge 5$, up to an additive constant.

For the remainder of this section, we fix $p \ge 5$. Let $\Sigma$ be an orientable
surface of
minimum genus in which $K_p$ embeds. The Euler genus of $\Sigma$ is at most $\frac{(p-3)(p-4)}{6}+1$ (see \cite{MT}).

We will use the following facts (see \cite{MT} for details):

\begin{Lemma}
\label{obs1}
There are at most $\frac{(p-3)(p-4)}{6}+1$ disjoint instances of $K_5$-minors in a graph which is embedded in the surface $\Sigma$.  Moreover,
suppose there are connected subgraphs $B_1,\dots,B_q$ in a graph embedded in the surface $\Sigma$,
such that
each $B_i$ contains a $K_5$-minor.  Assume there is a vertex $v$ such that
$v \in V(B_i)$ for each $i$ and $(V(B_i) - \{v\}) \cap (V(B_j) - \{v\}) = \emptyset$ for $i\not=j$.
Then $q \leq \frac{(p-3)(p-4)}{6}+1$.
\end{Lemma}

Lemma \ref{obs1} can be generalized as follows (again, see \cite{MT} for details):

\begin{Lemma}
\label{obs2}
Suppose there are $q$ disjoint minors isomorphic to $K_{l_1},K_{l_2},\dots,K_{l_q}$ ($l_i \geq 5$ for
$i=1,\dots,q$), respectively, in a graph $G$ that is embedded in the surface $\Sigma$.
Then $\Sigma_{i=1}^{q} \left\lceil\frac{(l_i-3)(l_i-4)}{6}\right\rceil \leq \frac{(p-3)(p-4)}{6} +1$.
Suppose there are connected graphs $B_1,\dots,B_q$ in a graph that is embedded into the surface $\Sigma$, such that
each $B_i$ contains a $K_{l_i}$-minor (with $l_i \geq 5$ for $i=1,\dots,q$), and there is a vertex $v$ such that
$v \in V(B_i)$ for each $i$ and $(V(B_i) - \{v\}) \cap (V(B_j) - \{v\}) = \emptyset$ for $i\not=j$.
Then $\Sigma_{i=1}^{q}\left \lceil\frac{(l_i-3)(l_i-4)}{6}\right\rceil \leq \frac{(p-3)(p-4)}{6} +1$.
\end{Lemma}

We are almost ready to construct the graph $G(n,k,p)$.  We first recall that the \emph{face-width} of a graph embedded in a surface is the minimum number of times a non-contractable loop intersects the embedded graph taken over all possible non-contractable loops.  The following observation follows immediately from the definition of face-width.
\begin{Observation}\label{obs3}
Let $G$ be a graph embedded in a surface $\Gamma$ with face-width $k$.  Let $X$ be a set of $t$ vertices in $G$.  Then $G-X$ is embedded in $\Gamma$ with face width at least $k-t$.
\end{Observation}For a further discussion of face-width, we refer to \cite{MT}.  We will need the following result.
\begin{Theorem}[\cite{RS7}]\label{thm:rs7}
Let $t \ge 5$ be a positive integer and let $\Gamma$ be a surface in which $K_t$ can be embedded.  Then there exists a value $r = r(\Gamma, t)$ such that every graph embedded in $\Gamma$ with face-width $r$ contains $K_t$ as a minor.
\end{Theorem}

Fix $r$ to be the value given by Theorem \ref{thm:rs7} to ensure a graph embedded in $\Sigma$ contains $K_p$ as a minor.  We first construct a graph $G'$ which is embedded
in the surface $\Sigma$, with the following properties:

\begin{enumerate}
\item The face-width of $G'$ embedded in $\Sigma$ is at least $n+r$.
\item There is a cycle $C$ in $G'$ which bounds a disk $D$ in $\Sigma$, and the set of vertices on the outer boundary of the disk
$D$ is defined by $V(D)$. We assume that no vertex, except for the vertex set $V(D)$, exists
inside the disk $D$.
\item
For each vertex $v$ outside the disk $D$, there are at least
$k(p-3) - \frac{(p-3)(p-4)}{2}-6$ internally disjoint paths from $v$ to $V(D)$ in $G'$.
\item
$G'$ is 3-connected, and hence each vertex in $V(D)$ has degree at least 3.
\end{enumerate}

A graph $G'$ with the desired embedding is known to exist \cite{mohar};
we outline such a construction.  We begin with a 3-connected graph $H$ allowing a \emph{closed 2-cell embedding} in $\Sigma$, in other words, a 3-connected graph $H$ which embeds in $\Sigma$ so that the topological closure of every facial region is homeomorphic to the closed disk.  Consider the following operation for a fixed facial region $F$.  The region $F$ is bounded by a cycle $C_F$ in $H$.  We subdivide every edge of $C$ and add a new vertex embedded in the region $F$ adjacent to every vertex on the subdivided cycle $C$.  The resulting graph is 3-connected and the new embedding is a closed 2-cell embedding as well.  Note that if we perform this operation on every facial region, the resulting graph will be embedded in $\Sigma$ with face width at least twice that of the original embedding.  Thus by repeatedly performing the operation, we find a $3$-connected graph $H_1$ along with a closed 2-cell embedding in $\Sigma$ satisfying 1 above.

Given the embedded graph $H_1$, let $H^*$ be the dual graph with vertex set equal to the set of facial regions and two facial regions are adjacent in $H^*$ if their boundary cycles share an edge.  Note that by the 3-connectivity of $H_1$, the graph $H^*$ is a simple connected graph.  Let $T$ be a spanning tree of $H^*$, and fix a root $R$ of the tree $T$.  Let $F \in V(T) -R$ be a facial region of $H_1$ forming a leaf in $T$.  Let $C_F$ be the boundary cycle of $F$, and let $e_F$ be the edge of $H_1$ shared with the neighboring facial region in $T$.  We subdivide the edge $e_F$ sufficiently many times to add $k(p-3) - \frac{(p-3)(p-4)}{2} - 6$ neighbors in the subdivided $e_F$ for every vertex of $C_F - e_F$.  Given that the region is homeomorphic to the disc, it is clear that we can add the edges maintaining the embedding in $\Sigma$.  Moreover, we maintain 3-connectivity of the graph.  In the resulting graph, every vertex of $C_F - e_F$ will have the desired large degree.  We repeatedly delete the leaf $F$ from the tree $T$ and apply the same process to a leaf of $V(T) - F$ until only the vertex $R$ remains.  Let $G'$ be the resulting graph.  We claim $G'$ satisfies $2-4$ above with the disc $D$ being the boundary cycle of the facial region $R$.  The properties $2$ and $4$ follow easily from the construction.  To see that we satisfy $3$ as well, pick a vertex of $v \in V(G') \setminus V(D)$.  We can find the desired paths from $v$ to $V(D)$ by looking at the path of facial regions in $T$ connecting $v$ to $V(D)$.   At each facial region along the path, a given vertex has in fact $k(p-3) - \frac{(p-3)(p-4)}{2} - 6$ neighbors on the next facial region.  This completes our outline of the construction of $G'$.

We now define $G = G(n, k, p)$ as follows.  Let $Z$ be a set of $k(p-3)-\frac{(p-3)(p-4)}{2}-6$ vertices.  The vertex set of $G$ will be $Z \cup V(G')$, and the edge set will the union of the edges of $G'$ along with every possible edge of the form $zd$ for all $z \in Z$ and $d \in D$.  We will see that the graph $G$ satisfies the desired properties.

We first claim that $G$ is $(k(p-3) - \frac{(p-3)(p-4)}{2}-6)$-connected.
Assume there exists a cutset $X \subseteq V(G)$ dividing the graph into at least two connected pieces with $X \le k(p-3) - \frac{(p-3)(p-4)}{2}-7$.  Let $u$ and $v$ be two vertices such that $u$ and $v$ are in distinct components of $G - X$.  There exists at least one element $z \in Z$ contained in $G - X$, and so it follows that $V(D) \setminus X$ is contained in a single component of $G-X$.  Given that there exist $k(p-3) - \frac{(p-3)(p-4)}{2}-6$ internally disjoint paths from each of $v$ and $u$ to $V(D)$, it follows that $V(D)\setminus X$, $u$, and $v$ are all contained in the same component of $G-X$, contrary to our choice of $u$ and $v$.

We now observe that there is no vertex set $X$ of order at most $n$ in $G$
such that $G-X$ does not contain a $K_p$-minor. By Observation \ref{obs3}, for any vertex set $X$ of order $n$,
the graph $G-X$ has face-width at least $r$.
It follows that $G-X$ contains $K_p$ as a minor by Theorem \ref{thm:rs7}.

As a final step, we now prove that $G$ cannot contain $k$ disjoint instances of $K_p$-minors
when $k \ge p$.
Suppose, to reach a contradiction, that $G$ contains pairwise disjoint subgraphs $H_1,\dots,H_k$, each of which contains $K_p$ as a minor.
Recall that $Z$ is the set of vertices adjacent every vertex of $D$, and $|Z|=k(p-3)-\frac{(p-3)(p-4)}{2}-6$.
We are now interested in all of the instances $H_1,\dots,H_k$
that contain at most $p-4$ vertices of $Z$. We fix the value $t$, and possibly re-number the subgraphs $H_i$ for $1\le i \le k$ such that $H_i$ contains at most $p-4$ vertices of $Z$ if and only if $1 \le i \le t$.  We let $l_i$ be defined to be $|V(H_i) \cap Z|$ for $1 \le i \le t$.  It follows immediately that:
\begin{equation*}
\begin{split}
|Z| & = \sum_1^t l_i + \sum_{i = t+1}^k |V(H_i) \cap Z| \\
& \ge \sum_1^t l_i + (k-t)(p-3) \\
& \ge \sum_1^t l_i + k(p-3) -tp +3t
\end{split}
\end{equation*}
If we combine the resulting inequality with the bound on $|Z|$, we conclude that
\begin{equation*}
\sum_1^t(p - l_i -3) \ge \frac{(p-3)(p-4)}{2}+6.
\end{equation*}

We now define a new graph $\overline{G}$ to be the graph $G'$ embedded in $\Sigma$ with an additional vertex $x$ attached to every vertex of $D$.  It is clear that the graph $\overline{G}$ embeds in $\Sigma$ as well.  We also define $\overline{H}_i$ for $1 \le i \le t$ to be the subgraph of $\overline{G}$ formed by $H_i \cap G'$ and the vertex $x$.  Observe that $\overline{H}_i$ contains a $p-l_i + 1$ clique minor.  This follows as every branch set of $H_i$ which does not intersect $Z$ remains a connected branch set of $\overline{H}_i$, and we form one additional branch set consisting of the union of the remaining branch sets of the clique minor in $H_i$ along with the vertex $x$.  Note that by our choice of $H_i$, $p-l_i +1 \ge 5$ for $1 \le i \le t$.

We now apply Lemma \ref{obs2} to the subgraphs $\overline{H}_i$, $1 \le i \le t$, of the graph $\overline{G}$.   It follows that:
\begin{equation*}
\begin{split}
\frac{(p-3)(p-4)}{6} + 1 & \ge \sum_1^t \left \lceil \frac{(p-l_i -2)(p-l_i-3)}{6} \right \rceil \\
& \ge \sum_1^t \left \lceil \frac{1}{3} (p-l_i - 3) \right \rceil \\
& \ge \sum_1^t  \frac{1}{3} (p-l_i - 3).
\end{split}
\end{equation*}
However, given our lower bound on $\sum_1^t (p-l_i -3)$, we now arrive at a contradiction.

This completes the proof that there exists a graph $G_{n,k,p}$ which is $(k(p-3)+\frac{(p-3)(p-4)}{2}-6)$-connected graph such that for
all integers $n \ge 1$, $p \ge 5$, $k \ge p$, the graph $G_{n, k, p}$
does not contain $k$ disjoint instances of $K_p$ as a minor, nor does it
contain a subset $X$ of vertices with $|X| \le n$ such that
$G(n,k,p) - X$ does not contain $K_p$ as a minor.

\end{document}